\documentclass[12pt]{article}

\usepackage{amssymb}%
\usepackage{amsmath}%
\usepackage{amsthm}%

\usepackage{fullpage}%
\usepackage{xcolor}%

\allowdisplaybreaks%

{\theoremstyle{plain}%
 \newtheorem{theorem}{Theorem}
 \newtheorem{corollary}{Corollary}

}
{\theoremstyle{remark}

}
{\theoremstyle{definition}

\newtheorem{example}{Example}
}

\begin{document}

\begin{center}
 {\large New Clebsch--Gordan-type integrals involving threefold products of complete elliptic integrals}

 \ 

 John M.\ Campbell 

 \ 

\end{center}

\begin{abstract}
 Multiple elliptic integrals related to the generalized Clebsch--Gordan (CG) integral are of importance in many areas in physics and special functions theory. 
 Zhou has introduced and applied Legendre function-based techniques to prove symbolic evaluations for integrals of CG form involving twofold and 
 threefold products of complete elliptic integral expressions, and this includes Zhou's remarkable proof of an open problem due to Wan. The foregoing 
 considerations motivate the results introduced in this article, in which we prove closed-form evaluations for new CG-type integrals that involve threefold 
 products of the complete elliptic integrals $\text{{\bf K}}$ and $\text{{\bf E}}$. Our methods are based on the use of fractional derivative operators, via 
 a variant of a technique we had previously referred to as \emph{semi-integration by parts}. 
\end{abstract}

\noindent {\footnotesize Keywords: multiple elliptic integral; complete elliptic integral; Clebsch--Gordan integral; closed form}

\section{Introduction}\label{sectionIntroduction}
 The expression \emph{multiple elliptic integral} refers to an integral involving an integrand factor given by the complete elliptic integral of the first kind 
\begin{equation}\label{Kdefinition}
 \text{{\bf K}}(k) := \int_{0}^{\frac{\pi}{2}} \left( 1 - k^2 \sin^2 \theta \right)^{-1/2} \, d\theta 
\end{equation}
 or the complete elliptic integral of the second kind 
\begin{equation}\label{Edefinition}
 \text{{\bf E}}(k) := \int_{0}^{\frac{\pi}{2}} \sqrt{1 - k^2 \sin^2 \theta} \, d\theta, 
\end{equation}
 up to a change of variables. Much about the study of multiple elliptic integrals is due to how mathematical objects of this form naturally arise within many 
 different areas of physics \cite{Kaplan1950}. For example, as expressed in \cite{Glasser1976}, practical problems concerning three-dimensional lattices 
 often give rise to triple integrals that are reducible to 
\begin{equation}\label{multipleK}
 \int_{\alpha}^{\beta} F(\mu) \text{{\bf K}}(\mu) \, d\mu 
\end{equation}
 for an elementary function $F$ \cite{Glasser1976}, and integrals as in \eqref{multipleK} involving twofold products 
 \cite{GlasserGuttmann1994,WanZucker2016} and threefold products \cite{WanZucker2016} of complete elliptic integral expressions have been similarly 
 applied in the context of the study of lattices. Zhou's 2014 article on multiple elliptic integrals \cite{Zhou2014Legendre}, in which Zhou solved the open 
 problem of proving the formulas 
\begin{equation}\label{openproblem}
 \int_{0}^{1} \text{{\bf K}}^{3}\left( \sqrt{1 - k^2} \right) \, dk = 6 \int_{0}^{1} \text{{\bf K}}^{2}(k) \text{{\bf K}}\left( \sqrt{1 - k^2} \right) k \, dk 
 = \frac{\Gamma^{8}\left( \frac{1}{4} \right)}{ 128 \pi^2} 
\end{equation}
 conjectured by Wan in 2012 \cite{Wan2012}, referenced applications of expressions as in \eqref{multipleK} within many different areas in physics, and 
 serves as main source of motivation behind the multiple elliptic integrals introduced in this article. A remarkable aspect of the definite integral evaluations 
 shown in \eqref{openproblem} is given by these integrals involving \emph{threefold} products of complete elliptic integral expressions, as opposed, for 
 example, to the many onefold or twofold products of $\text{{\bf K}}$- and/or $\text{{\bf E}}$-expressions in the integrands recorded in many of the 
 sections of a standard reference on multiple elliptic integrals \cite[\S4.21.1--4.22.14]{Brychkov2008}, noting that none of the integrals in 
 \cite{Brychkov2008} involve {threefold} products of elliptic integrals, with very little known about integrals involving such threefold products, relative to the 
 onefold or twofold cases. In our article, we introduce many new evaluations for definite integrals that resemble Wan's integrals in \eqref{openproblem} 
 and involve threefold products of complete elliptic expressions and that have not, to the best of our knowledge, appeared in any equivalent forms 
 in any relevant literature. In contrast to the Legendre function-based techniques introduced by Zhou in \cite{Zhou2014Legendre}, we instead make use 
 of Caputo operators, building on the recent work in \cite{Campbell2022,CampbellCantariniDAurizio2022}. 

 Zhou's Legendre polynomial-based approach from \cite{Zhou2014Legendre} was applied to prove the below closed form for what is referred to as a 
 \emph{multiple elliptic integral in Clebsch--Gordan (CG) form} \cite{Campbell2022,Zhou2014Legendre}, which was highlighted in Corollary 3.2 in 
 \cite{Zhou2014Legendre}, and which had been previously included in the standard reference on multiple elliptic integrals 
 previously referenced \cite[p.\ 278]{Brychkov2008}: 
\begin{equation}\label{maintwofold}
 2 \int_{0}^{1} \text{{\bf K}}\left( \sqrt{x} \right) \text{{\bf K}}\left( \sqrt{1 - x} \right) \, dx = \frac{\pi^3}{4}. 
\end{equation}
 The closed-form formula in \eqref{maintwofold} was later proved and generalized in \cite{Campbell2022} using a fractional-calculus based identity 
 introduced in \cite{CampbellCantariniDAurizio2022} and referred to as \emph{semi-integration by parts} (SIBP) \cite{CampbellCantariniDAurizio2022}. We 
 again emphasize the threefold nature of the products of complete elliptic functions in Wan's integrals in \eqref{openproblem}, in contrast to the twofold 
 product in the integrand in \eqref{maintwofold}. If the $F$-factor in \eqref{multipleK} is elementary, then Fubini-type interchanges often may be used to 
 evaluate \eqref{multipleK} by rewriting the $\text{{\bf K}}$-factor according to \eqref{Kdefinition}, which shows how this onefold case is much more 
 manageable relative to threefold products as in the Wan integrals shown in \eqref{openproblem}; a similar argument may be used to explain why 
 integrands with twofold products of complete elliptic integrals, as in the CG-type integral in \eqref{maintwofold}, are also relatively manageable compared 
 to the recalcitrant nature of integrals as in \eqref{openproblem}, noting that the challenging nature of the integrals in \eqref{openproblem} was also 
 considered by Wan and Zucker in \cite{WanZucker2016}. The foregoing considerations strongly motivate the exploration as to how the fractional 
 calculus-based approaches from \cite{Campbell2022,CampbellCantariniDAurizio2022} 
 may be improved or otherwise altered so as to be applicable to 
 integrals involving threefold products of $\text{{\bf K}}$ and/or $\text{{\bf E}}$. This is the main purpose of our article. 
 In this regard, we have succeeded in applying our new methods to 
 prove the new results highlighted in Section \ref{subsectionMotivating} below. 

\subsection{New multiple elliptic integrals}\label{subsectionMotivating}
 The main results in our article are given by how we have generalized the SIBP theorem from \cite{CampbellCantariniDAurizio2022}, together with our 
 applications of this generalization, as in the new closed forms highlighted below. 
 Zhou \cite{Zhou2014Legendre} has expressed and explored the analytically challenging nature of expressing integrals involving three or more complete 
 elliptic integrals in terms of fundamental mathematical constants, which  strongly motivates the new results highlighted 
 as \eqref{mainresult1}--\eqref{mainresult12}. With regard to our 
 new symbolic evaluation in \eqref{mainresult2}, we are letting $G = 1 - \frac{1}{3^2} + \frac{1}{5^2} - \cdots$
 denote Catalan's constant, and with reference to the new result in \eqref{mainresult3}, 
 we are letting $\phi = \frac{1 + \sqrt{5}}{2}$ denote the Golden Ratio constant. 
\begin{align}
 & \frac{\pi ^3 (1+4 \ln (2))}{32} 
 = \int_0^1 \text{{\bf E}}\left(\sqrt{1-x}\right) \text{{\bf K}}^{2}\left(\sqrt{\frac{1-\sqrt{1-x}}{2} }\right) \, dx, \label{mainresult1} \\
 & \ \nonumber \\
 & \frac{\pi ^2 (4 G+2+\pi \ln (2))}{16 \sqrt{2}}
 = \int_0^1 \text{{\bf E}}\left(\sqrt{1-x}\right) \text{{\bf K}}^2\left(\sqrt{\frac{1}{2}-\frac{1}{2} \sqrt{1-\frac{x}{2}}}\right) \, dx, \label{mainresult2} \\
 & \ \nonumber \\
 & \frac{\pi ^2}{2} \left(\frac{\pi ^2}{20}+\frac{3 \ln (\phi )}{2}-\frac{\sqrt{5}}{4}\right)
 = \int_0^1 \text{{\bf E}}\left(\sqrt{1-x}\right) \text{{\bf K}}^2\left(\sqrt{\frac{1}{2}-\frac{\sqrt{4+x}}{4}}\right) \, dx, \label{mainresult3} \\
 & \ \nonumber \\
 & \pi ^2 \left(\frac{17}{30}-\frac{\ln \left(1+\sqrt{2}\right)}{2 \sqrt{2}}\right)
 = \int_0^1 \frac{\text{{\bf E}}\left(\sqrt{1-x}\right)
 \text{{\bf K}}^{2}\left(\sqrt{\frac{1}{2}-\frac{\sqrt{\frac{1-\sqrt{1-x}}{x}}}{\sqrt{2}}}\right)}{\sqrt[4]{2 \sqrt{1-x} - x + 2}} \, dx, \label{mainresult4} \\
 & \ \nonumber \\
 & \frac{1}{{\sqrt[4]{2}}} \left( \frac{47 \pi ^2}{160}-\frac{\pi ^3}{16 \sqrt{3}} \right)
 = \int_0^1 \frac{\text{{\bf E}}\left(\sqrt{1-x}\right) 
 \text{{\bf K}}^2\left(\frac{\sqrt{4-\sqrt{6} \sqrt{\frac{\sqrt{16 x+9}-3}{x}}}}{2
 \sqrt{2}}\right)}{\sqrt[4]{8 x+3 \sqrt{16 x+9}+9}} \, dx, \label{mainresult5} \\
 & \ \nonumber \\
 & \frac{1}{2^{7/4}} \left( \frac{71 \pi ^2}{60}-\frac{\pi ^3}{8} \right) 
 = \int_0^1 \frac{\text{{\bf E}}\left(\sqrt{1-x}\right) \text{{\bf K}}^2\left(\sqrt{\frac{1}{2}-\frac{1}{4} \sqrt{\frac{\sqrt{8
 x+1}-1}{x}}}\right)}{\sqrt[4]{4 x+\sqrt{8 x+1}+1}} \, dx, \label{mainresult6} \\ 
 & \ \nonumber \\
 & \frac{\pi ^2 \left(143-\frac{20 \pi }{\sqrt{3}}\right)}{480 \sqrt[4]{2}} 
 = \int_0^1 \frac{\text{{\bf E}}\left(\sqrt{1-x}\right) \text{{\bf K}}^{2}\left(\sqrt{\frac{1}{2}-\frac{\sqrt{\frac{\sqrt{48 x+1}-1}{x}}}{4
 \sqrt{6}}}\right)}{\sqrt[4]{24 x+\sqrt{48 x+1}+1}} \, dx, \label{mainresult7} \\ 
 & \ \nonumber \\
 & \frac{\pi ^2 (104- 45 \ln (3))}{180 \sqrt{3}}
 = \int_0^1 \frac{\text{{\bf E}}\left(\sqrt{1-x}\right) \text{{\bf K}}^{2}\left(\frac{\sqrt{3-2 \sqrt{\frac{6-3 \sqrt{4-3
 x}}{x}}}}{\sqrt{6}}\right)}{\sqrt[4]{4 \sqrt{4-3 x}-3 x+8}} \, dx. \label{mainresult8} 
\end{align}
 We also introduce an infinite family of closed-form generalizations of \eqref{mainresult8}. 
 Our method allows us to obtain new evaluations such as 
 \begin{equation}\label{firstwithdE}
 \frac{\pi ^3}{64} (1-4 \ln (2)) = \int_{0}^{1} x \text{{\bf K}}^{2}\left( \sqrt{\frac{1 
 - \sqrt{x}}{2}} \right) \, d \text{{\bf E}}\left( \sqrt{x} \right), 
\end{equation}
 which may be written in an equivalent form so as to again obtain a threefold product of 
 complete elliptic expressions, recalling the differential relation such that $ \frac{d \text{{\bf E}}(k)}{dk} = \frac{ \text{{\bf E}}(k) - 
 \text{{\bf K}}(k) }{k}$. 
 In a similar vein, relative to \eqref{firstwithdE}, our method allows us to prove the following: 
\begin{align}
 & \frac{\pi ^2}{8} \left(3 \ln (\phi )-\frac{\sqrt{5}}{2}-\frac{\pi ^2}{10}\right) 
 = \int_{0}^{1} x 
 \text{{\bf K}}^2\left( \sqrt{\frac{1-\sqrt{\frac{1-x}{4}+1}}{2}} \right) \, 
 d \text{{\bf E}}\left( \sqrt{x} \right), \label{mainresult10} \\
 & \ \nonumber \\
 & 
 \frac{\pi ^2}{4} \left(\frac{\ln \left(1+\sqrt{2}\right)}{ \sqrt{2}}-\frac{13}{15}\right) 
 = \int_0^1 \frac{ x \text{{\bf K}}^{2}\left(\sqrt{\frac{1}{2}-\frac{1}{\sqrt{2} \sqrt{\sqrt{x}+1}}}\right) 
 }{\sqrt{\sqrt{x}+1}} \, d\text{{\bf E}}\left( \sqrt{x} \right), \label{mainresult11} \\
 & \ \nonumber \\
 & \frac{\pi ^2 \left(2 \ln (3)-\frac{152}{45}\right)}{16 \sqrt{3}}
 = \int_0^1 \frac{ x \text{{\bf K}}^{2}\left(\sqrt{\frac{1}{2}-\frac{\sqrt{\frac{\sqrt{3 x+1}-2}{x-1}}}{\sqrt{3}}}\right) 
 }{\sqrt[4]{3 x+4 
 \sqrt{3 x+1}+5}} \, d\text{{\bf E}}\left(\sqrt{x}\right). \label{mainresult12}
\end{align}
 Wan and Zucker \cite{WanZucker2016} 
 introduced a number of remarkable evaluations for 
 integrals that are of the forms suggested via \eqref{mainresult1}--\eqref{mainresult12}, 
 i.e., definite integrals satisfying the following properties: 
\begin{enumerate}
 \item The definite integral is from $0$ to $1$; 

 \item The integrand involves a factor given by a threefold product of complete elliptic integral expressions; and 

\item Any remaining integrand factors are algebraic. 
\end{enumerate}
 The closed-form evaluation of mathematical objects satisfying the above conditions is the main purpose of this article. 
 A remarkable evaluation for 
 an integral of this form 
 was given by Wan and Zucker \cite{WanZucker2016} 
 in the context of research on lattice sums and involves an integrand factor 
 of the form $2 \text{{\bf E}}(k)-\text{{\bf K}}(k)$. 
 We introduce, in Section \ref{subsectionWanZucker}, 
 new closed forms for integrals satisfying the above conditions 
 and involving a factor of the form $2 \text{{\bf E}}(k)-\text{{\bf K}}(k)$, inspired by \cite{WanZucker2016}. 

\section{Clebsch--Gordan theory}
 The CG coefficients are typically defined via the phenomenon of angular momentum coupling. Following \cite{Askey1982}, we express that the CG 
 coefficients that have zero magnetic quantum numbers satisfy the following identity: 
\begin{equation}\label{mainphysics}
 \left( C_{i0b0}^{c0} \right)^{2} = \frac{2c+1}{2} \int_{-1}^{1} P_{i}(x) P_{b}(x) P_{c}(x) \, dx, 
\end{equation}
 letting the orthogonal family of Legendre polynomials 
 be denoted as per usual. A common definition for Legendre polynomials 
 is via the binomial sum indicated as follows: $ P_{n}(x) = \frac{1}{2^{n}} \sum_{k=0}^{n} \binom{n}{k}^{2} (x+1)^{k} (x-1)^{n-k}$. 
 In view of \eqref{mainphysics}, and 
 as in the Zhou article \cite{Zhou2014Legendre} that is the main source of motivation behind our new results, 
 we may define \emph{generalized Clebsch--Gordan integrals} to be of the form 
\begin{equation}\label{generalizedCG}
 \int_{-1}^{1} P_{\mu}(x) P_{\nu}(x) P_{\nu}(-x) \, dx 
\end{equation}
 for $\mu$ and $\nu$ in $\mathbb{C}$, 
 and this terminology is also used in the article \cite{Cantarini2022} relevant to much of our work. 

 The CG coefficients naturally arise in both the decomposition of a product of two spherical harmonics into spherical harmonics and, equivalently, in the 
 decomposition of a product of Legendre polynomials into Legendre polynomials \cite{DongLemus2002}. So, by taking a product of three Legendre polynomials 
 and integrating this product, the CG coefficients naturally arise according to this latter decomposition. 
 From the product formula for $ Y_{\ell_{1}}^{m_{1}}(\theta, \varphi) Y_{\ell_{2}}^{m_{2}}(\theta, 
 \varphi)$ for spherical harmonics in terms of Legendre polynomials, we may write 
\begin{align}
 & P_{\ell_{1}}^{m_{1}}(x) P_{\ell_{2}}^{m_{2}}(x) = \label{prodPP1} \\ 
 & \sqrt{ \frac{ (\ell_{1} + m_{1})! (\ell_{2} + m_{2})! }{ (\ell_{1} - m_{1})! (\ell_{2} - m_{2})! } }
 \sum_{\ell_{12}} \sqrt{ \frac{ (\ell_{12} - m_{12})! }{ (\ell_{12} + m_{12})! } } 
 C_{m_{1}, m_{2}, m_{12}}^{\ell_{1}, \ell_{2}, \ell_{12}} C_{0, 0, 0}^{\ell_{1}, \ell_{2}, \ell_{12}}
 P_{\ell_{12}}^{m_{12}}(x), \label{prodPP2}
\end{align}
 referring to \cite{DongLemus2002} for details. 

 Integrals of threefold products of Legendre polynomials of the form shown in \eqref{generalizedCG} arise in the context of the evaluation of CG 
 coefficients in much the same way as in the classic identity in \eqref{mainphysics}, and hence the appropriateness as to how series and integral 
 evaluations arising from or otherwise directly relating to integrals as in \eqref{mainphysics} and \eqref{generalizedCG} may be referred to as being of CG 
 type, especially in view of the product identity shown in \eqref{prodPP1}--\eqref{prodPP2}. From Fourier--Legendre expansions such as 
\begin{equation}\label{FLofK}
 \text{{\bf K}}\left( \sqrt{x} \right) = \sum_{n=0}^{\infty} \frac{2}{2n+1} P_{n}\left( 2 x - 1 \right), 
\end{equation}
 the integration of products of complete elliptic-type expressions 
 often gives rise to CG coefficients via identities as in \eqref{prodPP1}--\eqref{prodPP2}. 

\section{New applications of semi-integration by parts}
 As indicated above, Zhou's 2014 article \cite{Zhou2014Legendre} is the main inspiration for our current work. This current work is also inspired by 
 many references citing or otherwise related to Zhou's article \cite{Zhou2014Legendre}, including 
 \cite{AusserlechnerGlasser2020,CampbellDAurizioSondow2019,Cantarini2022,GlasserZhou2018,RogersWanZucker2015,WanZucker2016,Zhou2022,Zhou2015,Zhou2017,Zhou2019,Zhou2014On,Zhou2016,Zhou2018}, 
 and these references include a number of articles involving definite integrals from $0$ to $1$ with integrands containing twofold products of complete 
 elliptic expressions \cite{AusserlechnerGlasser2020,CampbellDAurizioSondow2019,Cantarini2022,GlasserZhou2018,Zhou2017,Zhou2019} or threefold such 
 products \cite{RogersWanZucker2015,WanZucker2016,Zhou2014On}. 
 These references add to our interest in the multiple elliptic integrals highlighted in Section \ref{subsectionMotivating}
 and proved in the current Section. 

 A fundamental object in the field of fractional calculus is the \emph{Riemann--Liouville fractional derivative}, which is such that 
\begin{align*}
 D^{\alpha} f(x) & = \frac{d^{n}}{dx^{n}} \left( D^{-(n - \alpha)} f(x) \right) \\ 
 & = \frac{1}{\Gamma(n - \alpha)} \frac{d^{n}}{dx^{n}} \left( \int_{0}^{x} \left( x - t \right)^{n 
 - \alpha - 1} f(t) \, dt \right), 
\end{align*}
 setting $n - 1 \leq \alpha \leq n$ and $n \in \mathbb{N}$. 
 In this regard, and following \cite{Campbell2022,CampbellCantariniDAurizio2022}, the \emph{semi-derivative} operator 
 $D^{1/2}$ satisfies 
\begin{equation}\label{semiderivative}
 D^{1/2} x^{\alpha} = \frac{\Gamma(\alpha + 1)}{\Gamma\left( \alpha + \frac{1}{2} \right)} x^{\alpha - \frac{1}{2}}, 
\end{equation}
 and the \emph{semi-primitive} operator $D^{-1/2}$ satisfies 
\begin{equation}\label{semiprimitive}
 D^{-1/2} x^{\alpha} = \frac{\Gamma(\alpha + 1)}{\Gamma\left( \alpha + \frac{3}{2} \right)} x^{\alpha + \frac{1}{2}}, 
\end{equation}
 and we refer to the operators $D^{\pm 1/2}$ as \emph{Caputo operators}. As described in \cite{CampbellCantariniDAurizio2022}, much of the interest in 
 the techniques in \cite{CampbellCantariniDAurizio2022} involving the operators in \eqref{semiderivative} 
 and \eqref{semiprimitive} 
 may, informally, be regarded as being given by 
 how the application of $D^{\pm 1/2}$ 
 to series involving powers of central binomial coefficients 
 has the effect, by the Legendre duplication formula, of reducing such a power by one, 
 and this is often useful for the purposes of simplifying 
 series containing higher powers of $\binom{2n}{n}$ for $n \in \mathbb{N}_{0}$. 
 This is formalized, in part, in \cite{Campbell2022,CampbellCantariniDAurizio2022} with the following transformation. 

\begin{theorem}\label{SIBPtheorem}
 (Semi-integration by parts): The equality 
\begin{equation}\label{displayedSIBP}
 \langle f, g \rangle = \left\langle \left( D^{1/2} \tau \right) f, 
 \left( \tau D^{-1/2} \right) g \right\rangle 
\end{equation}
 holds true if both sides are well-defined, 
 and where $\tau$ maps a function $h(x)$ to $h(1-x)$ \cite{Campbell2022,CampbellCantariniDAurizio2022}. 
\end{theorem}

 From the Maclaurin series expansions 
\begin{equation*}
 \text{{\bf K}}(k) = \frac{\pi}{2} \, {}_{2}F_{1}\!\!\left[ 
 \begin{matrix} 
 \frac{1}{2}, \frac{1}{2} \vspace{1mm}\\  1 
 \end{matrix} \ \Bigg| \ k^2 \right] 
\end{equation*}
 and 
\begin{equation}\label{EMaclaurin}
 \text{{\bf E}}(k) = \frac{\pi}{2} \, {}_{2}F_{1}\!\!\left[ 
 \begin{matrix}
 \frac{1}{2}, -\frac{1}{2} \vspace{1mm}\\ 1 
 \end{matrix} \ \Bigg| \ k^2 \right], 
\end{equation}
 the term-by-term application of the Caputo operators  to the power series for $\text{{\bf K}}(\sqrt{k})$  and $\text{{\bf E}}(\sqrt{k})$ yields 
 elementary functions,  and, as explored in \cite{Campbell2022}, this is often useful in the evaluation and generalizations of integrals of CG form  as in 
 \eqref{maintwofold}.  In view of the power series expansions 
\begin{equation}\label{cubedcentralbinomial}
 \sum _{n=0}^{\infty} \binom{2 n}{n}^3 x^n
 = \frac{4 \text{{\bf K}}^{2}\left(\frac{\sqrt{1-\sqrt{1-64 x}}}{\sqrt{2}}\right)}{\pi ^2} 
\end{equation}
 and 
\begin{equation}\label{relatedcubedbinomial}
 \sum _{n=0}^{\infty} \binom{2 n}{n}^2 \binom{4 n}{2 n} x^n
 = \frac{4 \sqrt{2} \text{{\bf K}}^{2}\left(\sqrt{\frac{1}{2}-\frac{\sqrt{\frac{1-\sqrt{1-256 x}}{x}}}{16 \sqrt{2}}}\right)}{\pi ^2 \sqrt[4]{-256 x+2
 \sqrt{1-256 x}+2}}, 
\end{equation}
 we are led to consider something of an ``opposite'' strategy relative to \cite{CampbellCantariniDAurizio2022}, in the following sense: Instead of using 
 fractional derivatives to attempt to decrease the power of a central 
 binomial coefficient in a given series, we instead want to \emph{increase} 
 the power of $\binom{2n}{n}$, again with the use 
 of fractional derivatives, and by direct analogy with Theorem \ref{SIBPtheorem}. 
 We formalize this idea with Theorem \ref{SIBPvariant} below. 

\begin{theorem}\label{SIBPvariant}
 (A variant of SIBP for analytic functions): 
 For sequences $({a}_{n} : n \in \mathbb{N}_{0})$ and $({b}_{n} : n \in \mathbb{N}_{0})$, 
 we write ${f}(x) = \sum_{n=0}^{\infty} x^{n + \frac{1}{2}} {a}_{n} $ 
 and $\mathfrak{g}(x) = \sum_{n=0}^{\infty} x^{n+ \frac{1}{2}} {b}_{n}$. 
 The inner product 
\begin{equation}\label{displayinner1}
 \langle {f}(x), \mathfrak{g}(1-x) \rangle
\end{equation}
 may then be written as 
\begin{equation}\label{displayinner2}
 \left\langle \sum_{n=0}^{\infty} 
 \frac{\Gamma\left( n+\frac{3}{2} \right)}{\Gamma(n+1)} (1-x)^n {a}_{n}, 
 \sum_{n=0}^{\infty} \frac{\Gamma\left( n+\frac{3}{2} \right)}{\Gamma\left( n+2 \right)} 
 x^{n+1} {b}_{n} \right\rangle, 
\end{equation}
 under the assumption that applications of $\langle \cdot, \cdot \rangle$ and infinite summation 
 may be reversed in \eqref{displayinner1} and \eqref{displayinner2} (cf.\ \cite{Campbell2022}). 
\end{theorem}

\begin{proof}
 Under the given commutativity assumption, 
 it remains to consider the cases whereby $a_{n} = \delta_{n, \ell}$
 and $b_{n} = \delta_{n, m}$ for fixed $\ell, m \in \mathbb{N}_{0}$, letting the Kronecker delta 
 symbol be denoted as per usual. So, it remains to prove that 
 $$ \int_0^1 x^{\ell+\frac{1}{2}} (1-x)^{m+\frac{1}{2}} \, dx 
 = \int_{0}^{1} \left( \frac{(1-x)^\ell \Gamma \left(\ell+\frac{3}{2}\right)}{\Gamma (\ell+1)} \right) 
 \left( \frac{x^{m+1} \Gamma \left(m+\frac{3}{2}\right)}{\Gamma (m+2)} \right) \, dx. $$
 So, from the $\Gamma$-function evaluation of the beta integral, the desired result then immediately holds. 
\end{proof}

 Typically, to verify the required interchanges of limiting operations 
 specified in Theorem \ref{SIBPvariant}, one may employ 
 basic results in real analysis concerning term-by-term integration of infinite series \cite[\S5.3]{Bressoud2007}; 
 for the purposes of 
 our applications in Section \ref{subsectionApplications}, the required term-by-term integrations 
 may be justified in a routine way, 
 but obtaining closed forms from \eqref{displayinner1}
 can be quite involved and may require some degree of ingenuinty
 in the rare cases whereby the $a$- and $b$-sequences
 are such that \eqref{displayinner2}
 is expressible with a threefold product of complete elliptic integrals; this is clarified in Section \ref{subsectionApplications} below. 

\subsection{Applications}\label{subsectionApplications}
 The integral expressions given below are referred to as \emph{multiple elliptic integrals in CG form} by Zhou in \cite{Zhou2014Legendre}, and the following 
 evaluations due to Zhou \cite{Zhou2014Legendre} are highlighted as part of Corollary 2.2 in \cite{Zhou2014Legendre}. The following integrals in CG form 
 all involve threefold products of complete elliptic integrals, as in our new results listed in Section \ref{subsectionMotivating}, 
 and the below formulas due to Zhou, along with many results from Zhou of a similar quality, 
 are main sources of inspiration concerning our results as in 
 Section \ref{subsectionMotivating}: 
\begin{align}
 & 4 \int_{0}^{1} 
 \frac{ (1-t) \text{{\bf K}}^{2}\left( \sqrt{1-t} \right) \text{{\bf K}}\left( \sqrt{t} \right) }{ 
 (1 + t)^{3/2} } \, dt = \frac{ \Gamma^{2}\left( \frac{1}{8} \right) 
 \Gamma^{2}\left( \frac{3}{8} \right) }{24}, \label{nonintro1} \\
 & \frac{27}{4} \int_{0}^{1} 
 \frac{ t(1 - t) \text{{\bf K}}^{2}\left( \sqrt{1 - t} \right) \text{{\bf K}}\left( \sqrt{t}
 \right) }{ (1 - t + t^2)^{7/4} } \, dt
 = \frac{ \Gamma^{4}\left( \frac{1}{4} \right) }{ 8 \sqrt{2 \sqrt{3}}}. \label{nonintro2}
\end{align}
 Our present work is also inspired by Zhou's proof \cite{Zhou2014On} of the conjectured formula 
\begin{equation}\label{nonintro3}
 \int_{0}^{1} \frac{ \text{{\bf K}}^{2}\left( \sqrt{1 - k^2} \right) \text{{\bf K}}(k) }{\sqrt{k} 
 \left( 1 - k^2 \right)^{3/4}} \, dk 
 = \frac{ \Gamma^{8}\left( \frac{1}{4} \right) }{32 \sqrt{2} \pi^2} 
\end{equation}
 discovered experimentally by Rogers et al.\ \cite{RogersWanZucker2015}
 in the context of the study of formulas as in 
\begin{equation}\label{nonintro4}
 \int_{0}^{1} \frac{ \text{{\bf K}}^{3}\left( \sqrt{1 - k^2} \right) }{ \sqrt{k} 
 \left( 1 - k^2 \right)^{3/4} } \, dk = \frac{3 \Gamma^{8}\left( \frac{1}{4} \right) }{32 \sqrt{2} \pi^2}, 
\end{equation}
 as introduced in \cite{RogersWanZucker2015}.  The above results due to Zhou et al.\ as in \eqref{nonintro1}--\eqref{nonintro4}  motivate our first 
 Corollary to the SIBP variant formulated in Theorem \ref{SIBPvariant}, as below,  in view of our discussions in Section \ref{sectionIntroduction}.  As we 
 are to later explain,  the three integral formulas highlighted in the following Corollary are  ``special'' and ``non-arbitrary'' in the sense that  these 
 specific formulas depend on the very few known closed forms for the dilogarithm function. 

\begin{corollary}\label{corollaryfirstthree}
 The CG-type integral evaluations in \eqref{mainresult1}--\eqref{mainresult3} hold true. 
\end{corollary}

\begin{proof}
 We begin by setting $a_{n} = 4^{-n} \binom{2 n}{n}$ and  $b_{n} = \frac{16^{-n} (n+1) \binom{2 n}{n}^2}{2 n+1}$ in our SIBP variant given as 
 Theorem \ref{SIBPvariant}.  An application of Theorem \ref{SIBPvariant}, according to the specified input parameters,  then gives us the equality of 
\begin{align*}
 \int_{0}^{1} \frac{1}{12} \sqrt{x} \Bigg( 12 & \, \, {}_{3}F_{2}\!\!\left[ 
 \begin{matrix} 
 \frac{1}{2}, \frac{1}{2}, \frac{1}{2} \vspace{1mm}\\ 
 1, \frac{3}{2}
 \end{matrix} \ \Bigg| \ 1 - x \right] + \, {}_{3}F_{2}\!\!\left[ 
 \begin{matrix} 
 \frac{3}{2}, \frac{3}{2}, \frac{3}{2} \vspace{1mm}\\ 
 2, \frac{5}{2}
 \end{matrix} \ \Bigg| \ 1 - x \right] - \\
 & x \, {}_{3}F_{2}\!\!\left[ 
 \begin{matrix} 
 \frac{3}{2}, \frac{3}{2}, \frac{3}{2} \vspace{1mm}\\ 
 2, \frac{5}{2}
 \end{matrix} \ \Bigg| \ 1 - x \right] \Bigg) \, dx
\end{align*}
 and 
\begin{equation}\label{firstCGscalar}
 \int_0^1 \frac{2 \text{{\bf E}}\left(\sqrt{1-x}\right) \text{{\bf K}}^{2}\left(\sqrt{\frac{1}{2}-\frac{\sqrt{1-x}}{2}}\right)}{\pi ^2} \, dx, 
\end{equation}
 where the Maclaurin series in \eqref{EMaclaurin} and \eqref{cubedcentralbinomial} have been applied to obtain, via Theorem \ref{SIBPvariant}, the 
 integral in \eqref{firstCGscalar}. Applying term-by-term integration to the above expression involving ${}_{3}F_{2}$-series, this gives us that 
 \eqref{firstCGscalar} is equivalent to $$ \sum _{n=0}^{\infty } \left(\frac{1}{(2 n+1)^2}-\frac{1}{2 (2 n+1)} - \frac{1}{(2 n+3)^2}+\frac{3}{2 (2 
 n+3)}-\frac{1}{2 n+5}\right) \binom{2 n}{n} 2^{-2 n}. $$ Expanding the above summand, the resultant series are classically known, giving us a closed 
 form for \eqref{firstCGscalar} equivalent to \eqref{mainresult1}. The same approach as above, as applied to the cases such that $a_{n} = 4^{-n} 
 \binom{2 n}{n}$ and $b_{n} = \frac{ (n+1) \binom{2 n}{n}^2 \left( -\frac{1}{64} \right)^n}{2 n+1}$, may be used to prove the CG-type integral 
 evaluation in \eqref{mainresult3}. Similarly, setting $a_{n} = 4^{-n} \binom{2 n}{n}$ and $b_{n} = \frac{ (n+1) \binom{2 n}{n}^2 \left( \frac{1}{32} 
 \right)^n}{2 n+1}$ may be used to prove \eqref{mainresult2}. 
\end{proof}

 Setting $a_{n} = 4^{-n} \binom{2 n}{n}$ and $b_{n} = \frac{16^{-n} (n+1) \binom{2 n}{n}^2 \alpha ^n}{2 n+1}$
 for a free parameter $\alpha$ in Theorem \ref{SIBPvariant}, by mimicking our proof of 
 Corollary \ref{corollaryfirstthree}, 
 we can show that 
\begin{equation}\label{firstinfinite}
 \int_0^1 \text{{\bf E}}\left(\sqrt{1-x}\right) \text{{\bf K}}^{2}\left(\sqrt{\frac{1}{2}-\frac{1}{2} \sqrt{1-\alpha x}}\right) \, dx 
\end{equation}
 may be expressed as a combination of elementary functions, closed forms, together with the two-term dilogarithm combination 
\begin{equation}\label{twotermLi2}
 \text{Li}_2\left(-\sqrt{1-\alpha }-\sqrt{-\alpha }\right)-\text{Li}_2\left(-\sqrt{1-\alpha }-\sqrt{-\alpha }+1\right). 
\end{equation}
 So, the CG-type integral in \eqref{firstinfinite} admits a closed form if and only if the $\text{Li}_{2}$-combination in \eqref{twotermLi2} admits a closed 
 form. We have established a new connection between integrals related to CG theory and the closed-form evaluation of dilogarithmic expressions, and 
 it seems that there is not much known about connections of this form. Furthermore, past research articles exploring the closed-form evaluation of 
 two-term dilogarithm combinations as in \cite{Campbell2021Some,Khoi2014,Lima2012,Lima2017,Stewart2022} motivate our interest in the relationship 
 between \eqref{firstinfinite} and \eqref{twotermLi2} that we have introduced. By systematically setting each of the arguments of the two 
 $\text{Li}_{2}$-expressions in \eqref{twotermLi2} to be equal to the eight known real values $x$ such that both $\text{Li}_{2}(x)$ and $x$ admit 
 closed forms \cite[pp.\ 4, 6--7]{Lewin1981}, the only $\alpha$-value that yields a closed form with real arguments in \eqref{twotermLi2} is $\alpha 
 = -\frac{1}{4}$. A similar argument may be used to explain the uniqueness for the $\alpha = \frac{1}{2}$ case. A similar uniqueness property may be 
 applied to the results highlighted in Corollary \ref{corollarynextthree} below, as we later explain. 

\begin{corollary}\label{corollarynextthree}
 The CG-type integral evaluations in \eqref{mainresult4}--\eqref{mainresult7} hold true. 
\end{corollary}

\begin{proof}
 In Theorem \ref{SIBPvariant}, we set $a_{n} = 4^{-n} \binom{2 n}{n}$
 and $b_{n} = \frac{64^{-n} (n+1) \binom{2 n}{n} \binom{4 n}{2 n}}{2 n+1}$. 
 Using \eqref{relatedcubedbinomial}, Theorem \ref{SIBPvariant} then gives us the equality of 
\begin{align*}
 \int_{0}^{1} \frac{1}{16} \sqrt{x} \Bigg( 16 & \, 
 \, {}_{3}F_{2}\!\!\left[ 
 \begin{matrix} 
 \frac{1}{4}, \frac{1}{2}, \frac{3}{4} \vspace{1mm}\\ 
 1, \frac{3}{2}
 \end{matrix} \ \Bigg| \ 1 - x \right] + \, {}_{3}F_{2}\!\!\left[ 
 \begin{matrix} 
 \frac{5}{4}, \frac{3}{2}, \frac{7}{4} \vspace{1mm}\\ 
 2, \frac{5}{2}
 \end{matrix} \ \Bigg| \ 1 - x \right] - \\
 & x \, {}_{3}F_{2}\!\!\left[ 
 \begin{matrix} 
 \frac{5}{4}, \frac{3}{2}, \frac{7}{4} \vspace{1mm}\\ 
 2, \frac{5}{2}
 \end{matrix} \ \Bigg| \ 1 - x \right] \Bigg) \, dx
\end{align*}
 and 
\begin{equation}\label{multiplenotcubed}
 \int_0^1 \frac{2 \sqrt{2} \text{{\bf E}}\left(\sqrt{1-x}\right) 
 \text{{\bf K}}^{2}\left(\sqrt{\frac{1}{2}-\frac{\sqrt{\frac{1-\sqrt{1-x}}{x}}}{\sqrt{2}}}\right)}{\pi ^2 
 \sqrt[4]{-x+2 \sqrt{1-x}+2}} \, dx. 
\end{equation}
 Applying term-by-term integration to the above expression involving ${}_{3}F_{2}$-series, we find that the multiple elliptic integral in 
 \eqref{multiplenotcubed} equals $$ \sum _{n=0}^{\infty } \left(\frac{4}{(2 n + 1)^2} - \frac{33}{16 (2 n + 1)}-\frac{15}{4 (2 n + 3)^2} + 
 \frac{6}{2 n + 3} - \frac{63}{16 (2 n + 5)}\right) \binom{4 n}{2 n} 4^{-2 n - 1}. $$ This reduces, via classically known series expansions, to: 
 $$ \sum _{n=0}^{\infty } \frac{4^{-2 n} \binom{4 n}{2 n}}{(2 n + 1)^2} - \sum _{n=0}^{\infty } \frac{15\ 4^{-2-2 n} \binom{4 n}{2 n}}{(2 n + 
 3)^2}-\frac{3}{10 \sqrt{2}}. $$ So, it remains to evaluate 
\begin{equation}\label{3F24F3}
 {}_{3}F_{2}\!\!\left[ 
 \begin{matrix} 
 \frac{1}{4}, \frac{1}{2}, \frac{3}{4} \vspace{1mm}\\ 
 \frac{3}{2}, \frac{3}{2} 
 \end{matrix} \ \Bigg| \ 1 \right] \ \ \ \text{and} \ \ \ 
 {}_{4}F_{3}\!\!\left[ 
 \begin{matrix} 
 \frac{1}{4}, \frac{3}{4}, \frac{3}{2}, \frac{3}{2} \vspace{1mm}\\ 
 \frac{1}{2}, \frac{5}{2}, \frac{5}{2}
 \end{matrix} \ \Bigg| \ 1 \right]. 
\end{equation}
 For the former case, we may rewrite this ${}_{3}F_{2}(1)$-series as $$\int _0^1\int _0^1\frac{\sqrt{1+\sqrt{1-t^2 u^2}}}{\sqrt{2} \sqrt{1-t^2 u^2}} \, 
 dt \, du, $$ and the corresponding antiderivatives admit elementary forms. As for the ${}_{4}F_{3}(1)$-series, the same argument together with a 
 reindexing may be applied. 

 We proceed to set 
\begin{equation}\label{aandbalpha}
 a_{n} = 4^{-n} \binom{2 n}{n} \ \ \ \text{and} \ \ \ b_{n} = \frac{64^{-n} (n+1) \binom{2 n}{n} \binom{4 n}{2 n} \alpha ^n}{2 n+1} 
\end{equation}
 in Theorem \ref{SIBPvariant}. For the $\alpha = -\frac{16}{9}$ case, we may mimic our above proof to 
 prove \eqref{mainresult5}. For the $\alpha = -{8}$ case, 
 we may again mimic our above proof to prove \eqref{mainresult6}. 
 For the $\alpha = -{48}$ case, we may, once again, mimic our above proof to prove \eqref{mainresult7}. 
\end{proof}

 Setting $\alpha = \frac{3}{4}$ may be applied to prove \eqref{mainresult8}. The $\alpha = \frac{3}{4}$ case has led us to discover a complex analytic 
 property concerning the arctanh function that may be applied to prove the closed-form evaluation below for an infinite family of generalizations of 
 \eqref{mainresult8}, as in Corollary \ref{infinitecorollary} below. 

\begin{corollary}\label{infinitecorollary}
 The CG-type integral
\begin{equation}\label{infinitefamintegral}
 \int_0^1 \frac{\text{{\bf E}}\left(\sqrt{1-x}\right) \text{{\bf K}}^{2}\left(\sqrt{\frac{1}{2}-\frac{\sqrt{\frac{1-\sqrt{1-\alpha x}}{\alpha 
 x}}}{\sqrt{2}}}\right)}{\sqrt[4]{ + 2 \sqrt{1-\alpha x} - \alpha x + 2}} \, dx 
\end{equation}
 equals $$ \frac{\pi ^2 \left(36 \alpha +2 \sqrt{1-\alpha }-15 \sqrt{2} \sqrt{\frac{\sqrt{1-\alpha }+1}{\alpha }} \alpha \coth ^{-1}\left(\frac{\sqrt{2} 
 \sqrt{\sqrt{1-\alpha }+1}}{\sqrt{\alpha }}\right)-2\right)}{60 \sqrt{\sqrt{1-\alpha }+1} \alpha } $$ for positive values $\alpha$. 
\end{corollary}

\begin{proof}
 We again set the $a$- and $b$-sequences as in \eqref{aandbalpha}. Mimicking our proof of Corollary \ref{corollarynextthree}, we can show that 
 $$ \sum _{n=0}^{\infty } \left(\frac{4}{(2 n+1)^2} + \frac{-32-\alpha }{16 (2 n+1)} - \frac{15 \alpha }{4 (2 n+3)^2} + \frac{2 (1+2 \alpha )}{2 n + 
 3} - \frac{63 \alpha }{16 (2 n+5)}\right) \binom{4 n}{2 n} 2^{-2-4 n} \alpha ^n $$ equals 
\begin{equation}\label{infiniteclosed}
 \int_0^1 \frac{2 \sqrt{2} \text{{\bf E}}\left(\sqrt{1-x}\right) \text{{\bf K}}^{2}\left(\sqrt{\frac{1}{2}-\frac{\sqrt{\frac{1-\sqrt{1-\alpha x}}{\alpha 
 x}}}{\sqrt{2}}}\right)}{\pi ^2 \sqrt[4]{\alpha (-x)+2 \sqrt{1-\alpha x}+2}} \, dx 
\end{equation}
 for real values $\alpha$. 
 By direct analogy with our proof for 
 the ${}_{3}F_{2}$-series in \eqref{3F24F3}, we can show that 
 $$ {}_{3}F_{2}\!\!\left[ 
 \begin{matrix} 
 \frac{1}{4}, \frac{1}{2}, \frac{3}{4} \vspace{1mm}\\ 
 \frac{3}{2}, \frac{3}{2} 
 \end{matrix} \ \Bigg| \ \alpha \right] 
 = \frac{i \pi }{\sqrt{\alpha }}+\sqrt{2} \left(\frac{2}{\sqrt{\sqrt{1-\alpha }+1}}-\frac{2 \sqrt{2} \tanh ^{-1}\left(\frac{\sqrt{\alpha }+i
 \left(\sqrt{1-\alpha }+1\right)}{\sqrt{2} \sqrt{\sqrt{1-\alpha }+1}}\right)}{\sqrt{\alpha }}\right) $$
 for real $\alpha$. This can be used to show that 
 \eqref{infiniteclosed} is equal to the expression given by the following Mathematica output 
\begin{verbatim}
((-240*I)*Pi*Sqrt[1 + Sqrt[1 - \[Alpha]]]*Sqrt[\[Alpha]] - 
524*Sqrt[2]*\[Alpha] - 524*Sqrt[2 - 2*\[Alpha]]*\[Alpha] - 9*Sqrt[1 + 
Sqrt[1 - \[Alpha]]]*(-(Sqrt[2]*Sqrt[1 + Sqrt[1 - \[Alpha]]]) + 
2*Sqrt[2]*Sqrt[1 + Sqrt[1 - \[Alpha]]]*Sqrt[1 - \[Alpha]])*(2 + 2*Sqrt[1 - 
\[Alpha]] - \[Alpha])*\[Alpha] + (120*I)*Pi*Sqrt[1 + Sqrt[1 - 
\[Alpha]]]*\[Alpha]^(3/2) + 225*Sqrt[2]*\[Alpha]^2 - 45*Sqrt[2 - 
2*\[Alpha]]*\[Alpha]^2 + 18*Sqrt[2]*\[Alpha]^3 - (240*I)*Pi*Sqrt[1 + 
Sqrt[1 - \[Alpha]]]*Sqrt[-((-1 + \[Alpha])*\[Alpha])] - 480*Sqrt[1 + 
Sqrt[1 - \[Alpha]]]*(-2*Sqrt[\[Alpha]] + \[Alpha]^(3/2) - 2*Sqrt[-((-1 + 
\[Alpha])*\[Alpha])])*ArcTanh[(I*(1 + Sqrt[1 - \[Alpha]]) + 
Sqrt[\[Alpha]])/(Sqrt[2]*Sqrt[1 + Sqrt[1 - \[Alpha]]])])/(240*(-1 + Sqrt[1 - 
\[Alpha]])*(1 + Sqrt[1 - \[Alpha]])^(7/2))
\end{verbatim}
 So, it remains to evaluate 
\begin{equation}\label{maincomplex}
 \tanh ^{-1}\left(\frac{\sqrt{\alpha }+i \left(\sqrt{1-\alpha }+1\right)}{\sqrt{2} \sqrt{\sqrt{1-\alpha }+1}}\right)
\end{equation}
 for real variables $\alpha$. 
 For positive values $\alpha$, the usual 
 extension of the arctanh function for complex arguments gives us that \eqref{maincomplex} 
 equals 
\begin{align*}
 & -\frac{1}{4} \ln \left(\left(1-\frac{1}{\sqrt{2} \sqrt{\frac{\sqrt{1-\alpha }+1}{\alpha }}}\right)^2+\frac{1}{2} \left(\sqrt{1-\alpha 
 }+1\right)\right) + \\
 & \frac{1}{4} \ln \left(\left(\frac{1}{\sqrt{2} \sqrt{\frac{\sqrt{1-\alpha }+1}{\alpha }}}+1\right)^2+\frac{1}{2} 
 \left(\sqrt{1-\alpha }+1\right)\right)+ \\ 
 & i \left(\frac{1}{2} \tan ^{-1}\left(\frac{\sqrt{\sqrt{1-\alpha }+1}}{\sqrt{2} \left(1-\frac{1}{\sqrt{2} 
 \sqrt{\frac{\sqrt{1-\alpha }+1}{\alpha }}}\right)}\right)+\frac{1}{2} \tan ^{-1}\left(\frac{\sqrt{\sqrt{1-\alpha }+1}}{\sqrt{2} 
 \left(\frac{1}{\sqrt{2} \sqrt{\frac{\sqrt{1-\alpha }+1}{\alpha }}}+1\right)}\right)\right) 
\end{align*}
 Differentiating the expression 
 $$ \frac{1}{2} \tan ^{-1}\left(\frac{\sqrt{\sqrt{1-\alpha }+1}}{\sqrt{2} \left(1-\frac{1}{\sqrt{2} \sqrt{\frac{\sqrt{1-\alpha }+1}{\alpha
 }}}\right)}\right)+\frac{1}{2} \tan ^{-1}\left(\frac{\sqrt{\sqrt{1-\alpha }+1}}{\sqrt{2} \left(\frac{1}{\sqrt{2} \sqrt{\frac{\sqrt{1-\alpha
 }+1}{\alpha }}}+1\right)}\right)$$
 can be used to show that the abvoe expression always equals $\frac{\pi }{4}$, 
 and this gives us our desired closed form for the $\alpha > 0$ case. 
\end{proof}

 If $\alpha < 0$, the evaluation of \eqref{maincomplex} proves to be of a more interesting nature in terms of the challenge of producing closed forms as 
 in Corollary \ref{corollarynextthree}. This is formalized as below. For the time being, we remark that the arccoth evaluation in Corollary 
 \ref{infinitecorollary} is of interest in terms of how this evaluation may be used to obtain simply closed forms for CG-type integrals, by sysmetically 
 searching for $\alpha$-values such that the arccoth argument in Corollary \ref{infinitecorollary} is reducible to simple constants such as $\ln 2$. 

\begin{example}
 Setting $\alpha = \frac{576}{625}$, we can prove the evaluation 
 $$ \frac{\pi ^2 (551-400 \ln (2))}{3840 \sqrt{2}}
 = \int_0^1 \frac{\text{{\bf E}}\left(\sqrt{1-x}\right) \text{{\bf K}}^2\left(\frac{\sqrt{24-5 \sqrt{\frac{50-2 \sqrt{625-576 x}}{x}}}}{4
 \sqrt{3}}\right)}{\sqrt[4]{50 \left(\sqrt{625-576 x}+25\right)-576 x}} \, dx. $$
\end{example}

\begin{example}
 Setting $\alpha = \frac{32}{81}$, we can prove the evaluation 
 $$ \pi ^2 \left(\frac{93}{640}-\frac{3 \ln (2)}{32}\right)
 = \int_0^1 \frac{\text{{\bf E}}\left(\sqrt{1-x}\right) \text{{\bf K}}^{2}\left(\sqrt{\frac{1}{2}-\frac{3}{8} \sqrt{\frac{9-\sqrt{81-32
 x}}{x}}}\right)}{\sqrt[4]{18 \left(\sqrt{81-32 x}+9\right)-32 x}} \, dx. $$
\end{example}
 
\begin{corollary}\label{classificationcorollary}
 For $\alpha < 0$, the integral in \eqref{infinitefamintegral} is expressible in closed
 form as a finite combination of algebraic expressions and a given set of previously recognized constants 
 if and only if the same applies to 
\begin{equation}\label{mainarctan}
 \tan ^{-1}\left(\frac{\sqrt{1-\alpha }+\frac{1}{\sqrt{-\frac{1}{\alpha }}}+1}{\sqrt{2} \sqrt{\sqrt{1-\alpha }+1}}\right). 
\end{equation}
\end{corollary}

\begin{proof}
 From the evaluation for \eqref{infiniteclosed} given in the proof of Corollary \ref{infinitecorollary}, this gives us that the desired integral in 
 \eqref{infinitefamintegral} equals a finite combination of algebraic expressions together with $\pi^2$ and \eqref{maincomplex}, for all real values 
 $ \alpha$. For negative values $\alpha$, the usual extension of the arctanh function to complex arguments gives us that \eqref{maincomplex} equals the 
 imaginary unit times the arctan expression in \eqref{mainarctan}, and hence the desired result. 
\end{proof}

 The classification result in Corollary \ref{classificationcorollary} is of interest in the following sense: If we want to obtain a closed form for the CG-type 
 integral in \eqref{infinitefamintegral} for rational values $\alpha$, then it is only in exceptional cases that \eqref{mainarctan} will admit a closed form. 
 A systematic computer search based on the algebraic values of $\tan(q \pi)$ for $q \in \mathbb{Q}$ further demonstrated that \eqref{mainarctan} is 
 expressible in closed form for $\alpha \in \mathbb{Q}$ in only exceptional cases, which emphasizes the unique quality of the CG-type integrals 
 highlighted in \eqref{mainresult4}--\eqref{mainresult7}. For example, setting $\alpha = -\frac{1}{3}$ yields an expression involving 
 $$ \tan^{-1}\left(\sqrt{\sqrt{3}-\frac{3}{2}} \left(1+\sqrt{3}\right)\right),$$ which does not seem to be reducible, e.g., to a rational multiple of $\pi$ or 
 otherwise. The $\alpha = -\frac{16}{9}$ case corresponds to the closed form $\tan \left(\frac{\pi }{3}\right) = \sqrt{3}$, and similarly for 
 \eqref{mainresult6} and \eqref{mainresult7}. 

 The integral evaluations from Section \ref{subsectionMotivating} listed as \eqref{firstwithdE}--\eqref{mainresult12}
 may be proved via Theorem \ref{SIBPvariant} 
 in much the same way as in with Corollaries \ref{corollaryfirstthree}--\ref{infinitecorollary}; 
 for the sake of brevity, we leave it to the reader to verify this. 

\subsection{New integrals inspired by Wan and Zucker}\label{subsectionWanZucker}
 Integrals involving threefold products of complete elliptic functions 
 as in the formulas 
\begin{equation}\label{nonintro5}
\frac{\pi ^3}{6 \sqrt{2}} = 
 \int_0^1 \sqrt{\frac{k}{\sqrt{1-k^2}}} \text{{\bf K}}^{2}\left(\sqrt{1-k^2}\right)
 (2 \text{{\bf E}}(k)-\text{{\bf K}}(k)) \, dk
\end{equation}
 and 
\begin{equation}\label{nonintro6}
 \frac{\Gamma^4 \left(\frac{1}{8}\right) 
 \Gamma^4 \left(\frac{3}{8}\right)}{384 \sqrt{2} \pi ^2} 
 = \int_0^1 \frac{\left(2 + 3 k - k^2\right) \text{{\bf K}}^3(k)}{\sqrt{k+1}} \, dk 
\end{equation}
 and 
\begin{equation}\label{nonintro7}
 \frac{\left(\sqrt{2}-1\right)^{3/2} \Gamma^8 \left(\frac{1}{4}\right)}{128 \sqrt{2} \pi ^2} 
 = \int_0^1 \sqrt[4]{k} \sqrt[4]{1-k^2} \text{{\bf K}}^3(k) \, dk
\end{equation}
 were given by Wan and Zucker in \cite{WanZucker2016} in the context of the study of lattice sums, and this inspires the new results we introduce below, 
 which resemble the Wan--Zucker formula in \eqref{nonintro5}, as we provide closed forms resembling the left-hand side of \eqref{nonintro5} for integrals 
 satisfying the three listed conditions in Section \ref{subsectionMotivating} together with the condition that the integrands are to contain the last integrand 
 factor displayed in \eqref{nonintro5}. 

\begin{corollary}\label{corollarylattice}
 The CG-type integral evaluations below hold: 
\begin{align*}
 & \frac{\pi ^3 (29+32 \ln (2))}{2048} 
 = \int_0^1 y \left(1-y^2\right) \text{{\bf K}}^{2}\left(\sqrt{\frac{1}{2}-\frac{y}{2}}\right)
 (2 \text{{\bf E}}(y)-\text{{\bf K}}(y)) \, dy, \\
 & \ \\ 
 & \frac{\pi ^2 G}{32 \sqrt{2}} + 
 \frac{7 \pi ^2}{64 \sqrt{2}}-\frac{9 \pi ^3}{512 \sqrt{2}}+\frac{\pi ^3 \ln (2)}{128 \sqrt{2}} = \\
 & \int_0^1 y \left(1 - 
 y^2\right) \text{{\bf K}}^{2}\left(\frac{\sqrt{2-\sqrt{2} \sqrt{y^2+1}}}{2} \right)
 (2 \text{{\bf E}}(y)-\text{{\bf K}}(y)) \, dy, \\
 & \ \\
 & -\frac{57}{64} \pi ^2 \ln (\phi )+\frac{\pi ^4}{320}+\frac{53 \sqrt{5} \pi ^2}{256} = \\ 
 & \int_0^1 y \left(1-y^2\right) \text{{\bf K}}^{2}\left(\frac{\sqrt{2-\sqrt{5-y^2}}}{2} \right)
 (2 \text{{\bf E}}(y)-\text{{\bf K}}(y)) \, dy.
\end{align*}
\end{corollary}

\begin{proof}
 We set 
\begin{equation}\label{latticeinput}
 a_{n} = \frac{4^{-n} \binom{2 n}{n}}{2 n-1} \ \ \ 
 \text{and} \ \ \ b_{n} = \frac{16^{-n} (n+1) \binom{2 n}{n}^2}{2 n+1} 
\end{equation}
 in Theorem \ref{SIBPvariant}. 
 This gives us the equality of 
\begin{align*}
 \int_{0}^{1} -\frac{1}{12} (1-x) \sqrt{x} \Bigg( 12 & \, 
 \, {}_{3}F_{2}\!\!\left[ 
 \begin{matrix} 
 \frac{1}{2}, \frac{1}{2}, \frac{1}{2} \vspace{1mm}\\ 
 1, \frac{3}{2}
 \end{matrix} \ \Bigg| \ 1 - x \right] + \, {}_{3}F_{2}\!\!\left[ 
 \begin{matrix} 
 \frac{3}{2}, \frac{3}{2}, \frac{3}{2} \vspace{1mm}\\ 
 2, \frac{5}{2}
 \end{matrix} \ \Bigg| \ 1 - x \right] - \\
 & x \, {}_{3}F_{2}\!\!\left[ 
 \begin{matrix} 
 \frac{3}{2}, \frac{3}{2}, \frac{3}{2} \vspace{1mm}\\ 
 2, \frac{5}{2}
 \end{matrix} \ \Bigg| \ 1 - x \right] \Bigg) \, dx
\end{align*}
 and $$ \int_0^1 -\frac{2 x \text{{\bf K}}^{2}\left(\sqrt{\frac{1}{2}-\frac{\sqrt{1-x}}{2}}\right) \left(2 \text{{\bf E}}\left(\sqrt{1-x}\right)-\text{{\bf 
 K}}\left(\sqrt{1-x}\right)\right)}{\pi ^2} \, dx. $$ Reversing integration and infinite summation in the above expression involving 
 ${}_{3}F_{2}$-functions, we obtain the expression 
\begin{align*}
 & \sum _{n=0}^{\infty } \Bigg(-\frac{1}{8 (2 n+1)^2}-\frac{1}{32 (2 n+1)}+\frac{1}{8 (2 n+3)^2} + \\
 & \frac{3}{32 (2 n+3)}-\frac{11}{32 (2 n+5)} + \frac{9}{32 (2 n+7)} \Bigg) \binom{2 n}{n} 2^{1-2 n}, 
\end{align*}
 which reduces to $\frac{1}{512} (-29 \pi -32 \pi \ln (2))$ according to classically known series. 
 So, a change of variables then gives us the desired closed form for 
 $$ \int_0^1 y \left(1-y^2\right) \text{{\bf K}}^{2}\left(\sqrt{\frac{1}{2}-\frac{y}{2}}\right)
 (2 \text{{\bf E}}(y)-\text{{\bf K}}(y)) \, dy,$$ 
 and similarly for the remaining integrals in the Corollary under consideration. 
\end{proof}

 By setting $a_{n} = \frac{4^{-n} \binom{2 n}{n}}{2 n-1}$ $b_{n} = \frac{16^{-n} (n+1) \binom{2 n}{n}^2}{2 n+1} \alpha^{n}$ in Theorem 
 \ref{SIBPvariant}, the resultant CG-type integral can be shown to be reducible to a combination of elementary functions together with the same 
 two-term dilogarithm combination in \eqref{twotermLi2}. So, we may repeat a previous argument to explain the uniqueness of the closed forms in 
 Corollary \ref{corollarylattice}. Using variants of the input sequences in \eqref{latticeinput}, we may obtain many further integrals involving $2 \text{{\bf 
 E}}(y) - \text{{\bf K}}(y)$, and we encourage the exploration of this. 

\section{Conclusion}
 We conclude by briefly considering some subjects of further research connected with the above material. 

 Although our article is mainly devoted to integrals involving threefold products of products of $\text{{\bf K}}$ and/or $\text{{\bf E}}$, a careful 
 examination of Zhou's 2014 article \cite{Zhou2014Legendre} relative to the main material in this article and relative to material concerning integrals 
 involving $\text{{\bf K}}$ and/or $\text{{\bf E}}$ from relevant references such as \cite{Campbell2021New,CampbellDAurizioSondow2019} leads to new 
 results on CG-type integrals that involve twofold products and that are related to what Zhou refers to as \emph{Ramanujan transformations} 
 \cite{Zhou2014Legendre}. By ``reverse-engineering'' formulas from Ramanujan's notebooks that were presented without proof, Zhou 
 \cite{Zhou2014Legendre} showed that 
\begin{equation}\label{Ramanujantransform1}
 \text{{\bf K}}^{2}\left( \sqrt{t} \right) = \frac{2}{\pi} \int_{0}^{1} 
 \frac{ \text{{\bf K}}\left( \sqrt{\mu} \right) \text{{\bf K}}\left( \sqrt{1 - \mu} \right) }{1 - \mu t} \, d\mu 
\end{equation}
 and that 
\begin{equation}\label{Ramanujantransform2}
 \text{{\bf K}}^{2}\left( \sqrt{1 - t} \right) 
 = \frac{8}{\pi} \int_{0}^{1} \frac{ \text{{\bf K}}\left( \sqrt{\mu} \right) 
 \text{{\bf K}}\left( \sqrt{1 - \mu} \right) }{ (1 + \sqrt{t})^{2} - \mu (1 - \sqrt{t})^{2} } \, d\mu, 
\end{equation}
 and these Ramanujan-inspired transforms may be used to generalize the CG-type integral in 
 \eqref{maintwofold}, in the following manner, using material from 
 \cite{Campbell2021New,CampbellDAurizioSondow2019}. 

 The closed-form evaluation 
\begin{equation}\label{Polishformula}
 \sum_{i, j = 0}^{\infty} \frac{ \binom{2i}{i}^2 \binom{2j}{j}^2 }{4^{2i+2j} (i + j + 1)} 
 = \frac{14 \zeta(3)}{\pi^2} 
\end{equation}
 introduced in \cite{Campbell2021New} can be shown to be equivalent to a corresponding evaluation for 
\begin{equation}\label{Polishequivalent}
 \frac{7 \zeta (3)}{2} = \int_{0}^{1} \text{{\bf K}}^{2}\left( \sqrt{t} \right) \, dt, 
\end{equation}
 through the use of Fourier--Legendre expansion in \eqref{FLofK}, and this approach was generalized and explored in \cite{Campbell2021New} through the 
 use of the moment formula for Legendre polynomials. So, from the evaluation in terms of Ap\'{e}ry's constant for \eqref{Polishequivalent}, by applying the 
 operator $\int_{0}^{1} \cdot \, dt $ to both sides of \eqref{Ramanujantransform1} and then applying Fubini's theorem, this gives us an evaluation for 
\begin{equation}\label{generalizetwofold}
 -\frac{7 \pi \zeta (3)}{4} = \int_{0}^{1} 
 \frac{\ln (1-\mu )}{\mu } \text{{\bf K}}\left( \sqrt{\mu} \right) \text{{\bf K}}\left( \sqrt{1 - \mu} \right) \, d\mu, 
\end{equation}
 which was proved in a different way by Wan in \cite{Wan2012}. By mimicking the above described approach for evaluating 
 \eqref{generalizetwofold}, with the use of the many infinite families of double sums generalizing or otherwise related to \eqref{Polishformula} given 
 in \cite{Campbell2021New} and the follow-up article \cite{ChuCampbell2023} greatly generalizing the techniques from \cite{Campbell2021New}, this 
 leads us to new families of evaluations for CG-type integrals of the form 
\begin{equation}\label{generalizetwofold}
 \int_{0}^{1} F(\mu) \text{{\bf K}}\left( \sqrt{\mu} \right) \text{{\bf K}}\left( \sqrt{1 - \mu} \right) \, d\mu
\end{equation}
 for elementary functions $F(\mu)$, 
 and similarly with respect to the Ramanujan transform in \eqref{Ramanujantransform2}. 
 We leave it to a separate project to pursue a full exploration of this. 

 Recall that our SIBP variant presented as Theorem \ref{SIBPvariant} involved the series $${f}(x) = \sum_{n=0}^{\infty} x^{n + \frac{1}{2}} {a}_{n} \ \ 
 \ \text{and} \ \ \ \mathfrak{g}(x) = \sum_{n=0}^{\infty} x^{n+ \frac{1}{2}} {b}_{n}.$$ A fruitful area of research to explore would involve the 
 application of further variants of SIBP based on the series obtained by replacing $x_{n + \frac{1}{2}}$ in the expansions for $f(x)$ and 
 $\mathfrak{g}(x) $, respectively, with $x^{n+h_{1}}$ and $x^{n + h_{2}}$ for half-integer parameters $h_{1}$ and $h_{2}$. In a similar spirit, the 
 SIBP formula in \eqref{displayedSIBP} may be extended using operators other than $D^{\pm 1/2}$ such as $D^{\pm 1/4}$. 

\subsection*{Acknowledgements}
 The author is very grateful to Dr.\ Yajun Zhou for many useful discussions related to the results introduced in this article. 

\subsection*{Competing interests statement}
 There are no competing interests to declare.

 \

John M.\ Campbell

Department of Mathematics

Toronto Metropolitan University

350 Victoria St, Toronto, ON M5B 2K3

 \ 

{\tt jmaxwellcampbell@gmail.com}

\end{document}